\documentclass[12pt]{article}%

\usepackage{amsmath,enumerate}
\usepackage{amsfonts}
\usepackage{amssymb}

\newtheorem{theorem}{Theorem}[section]

\newtheorem{conjecture}[theorem]{Conjecture}

\newtheorem{lemma}[theorem]{Lemma}

\newtheorem{proposition}[theorem]{Proposition}

\newenvironment{proof}[1][Proof]{\noindent\textbf{#1.} }
{\hfill \ \rule{0.5em}{0.5em}}

\newcommand{\nn}{\mathbb{N}}

\newcommand{\exf}{\textup{ex}(n,F)}

\title{Infinite Tur\'an problems for  bipartite graphs}
\author{Xing Peng
\thanks{Department of Mathematics, University of California, San Diego,  La Jolla, CA 92093,
({\tt x2peng@ucsd.edu}). }
\and
Craig Timmons
\thanks{Department of Mathematics, University of California, San Diego,  La Jolla, CA 92093,
({\tt ctimmons@ucsd.edu}).
}
}
\date{}
\begin{document}
\maketitle

\begin{abstract}

We consider an infinite version of the bipartite Tur\'{a}n problem.  Let $G$ be an infinite graph with
$V(G) = \mathbb{N}$ and let $G_n$ be the $n$-vertex subgraph of $G$ induced by the vertices $\{1,2, \dots , n \}$.  We show that if $G$ is $K_{2,t+1}$-free then for infinitely many $n$, $e(G_n) \leq 0.471 \sqrt{t} n^{3/2}$.  Using the $K_{2,t+1}$-free graphs constructed by F\"{u}redi, we construct an infinite $K_{2,t+1}$-free graph with $e(G_n) \geq 0.23 \sqrt{t}n^{3/2}$ for all $n \geq n_0$.

\end{abstract}

\section{Introduction}

Given a graph $F$, a graph $G$ is \emph{$F$-free} if $G$ does not contain $F$ as a subgraph.
The Tur\'{a}n number of $F$, denoted $\textup{ex}(n, F)$, is the maximum number of edges in an
$n$-vertex graph that is $F$-free.
Finding good estimates on $\exf$ for different graphs $F$ is a well studied problem in extremal combinatorics.  The famous
Erd\H{o}s-Stone-Simonovits Theorem gives an asymptotic formula for $\exf$ when $F$ is not bipartite:
\[
\textup{ex}(n , F) = \left(1 - \frac{1}{ \chi (F) - 1} \right) \binom{n}{2} + o(n^2).
\]
If $F$ is bipartite then there is no general asymptotic formula for $\exf$.  In particular, if $F$ is a bipartite graph that contains a cycle then it can be very difficult to determine the order of magnitude of $\exf$.
For example it is still an open problem to determine the order of magnitude of $\textup{ex}(n , C_{2k})$ whenever $k \notin \{2,3,5 \}$.

In this paper we consider an infinite version of the bipartite Tur\'{a}n problem.  Infinite Tur\'{a}n problems for ordered paths
were studied by Czipszer, Erd\H{o}s, and Hajnal \cite{ceh} and Dudek and R\"{o}dl \cite{dr}.  We will briefly discuss a few of their results as it serves as motivation for the way in which we define the infinite Tur\'{a}n number of a graph.

Let $G$ be an infinite graph with $V(G) = \mathbb{N}$.  Write $G_n$ for the subgraph of
$G$ induced by the vertices $\{1,2, \dots , n \}$.  An \emph{increasing path of length $k$}, denoted $I_k$, is a
set of $k+1$ vertices $\{n_1 , n_2 , \dots , n_{k+1} \}$ such that $n_i < n_{i+1}$ and $n_i$ is adjacent to $n_{i+1}$ for $1 \leq i \leq k$.  Let $\mathcal{G}_{I_{k}}$ be the family of all infinite $I_k$-free graphs $G$ with $V(G) = \mathbb{N}$ and
define
\[
p(k) = \sup_{G \in \mathcal{G}_{I_k}  } \left\{ \liminf_{n \rightarrow \infty} \frac{ e(G_n)}{n^{2}} \right\}.
\]
The function $p(k)$ was introduced in \cite{ceh} and the notation $p(k)$ was first used in \cite{dr}.

For $k \geq 2$,  let $T_k(\infty)$ be the infinite graph with vertex set $\nn$ and  two vertices
$i$ and $j$ are adjacent if $i \not \equiv j (\textup{mod}~k)$.
Let $T_{k}' ( \infty)$ be the infinite graph with vertex set $\mathbb{N}$ where $s$ is adjacent to $t$ if $s<t$,
$s \equiv i (\textup{mod}~k)$, $t \equiv j (\textup{mod}~k)$, and $1 \leq i < j \leq k$.
The graph $T_{k}' ( \infty)$ is $I_k$-free so
$p(k) \geq \frac{1}{4}\left(1 - \frac{1}{k} \right)$ for all $k \geq 2$.  This construction is given in \cite{ceh} where it is also shown that $p(k) = \frac{1}{4} \left( 1- \frac{1}{k} \right)$ for $k \in \{2,3 \}$.
Czipszer et.\ al.\ asked if this formula holds for all $k$.  Dudek and R\"{o}dl \cite{dr} answered this question in the negative by showing that
$p(16) > \frac{1}{4} \left( 1 - \frac{1}{16} \right)$ and that for $k \geq 162$, $p(k) > \frac{1}{4} + \frac{1}{200}$.  It is an open problem to determine whether or not $p(k) = \frac{1}{4} \left( 1 - \frac{1}{k} \right)$ holds for $k \in \{4,5, \dots , 15 \}$.

Here we investigate infinite Tur\'{a}n problems where the ordering the vertices in the forbidden graph is not specified.
For non-bipartite graphs the infinite version of the problem that we consider is not  interesting.
Given a graph $F$, let $\mathcal{G}_F$ be the family of all infinite $F$-free graphs $G$ with $V(G) = \mathbb{N}$.
\begin{proposition}
If $F$ is a non-bipartite graph with $\chi(F) = r$ then
\[
\sup_{G \in \mathcal{G}_F} \left\{ \liminf_{n \rightarrow \infty} \frac{ e(G_n)}{n^{2}} \right\} =
\sup_{G \in \mathcal{G}_F} \left\{ \limsup_{n \rightarrow \infty} \frac{ e(G_n)}{n^{2}} \right\} =
\frac{1}{2} \left( 1 - \frac{1}{r-1} \right).
\]
\end{proposition}
\begin{proof}
The upper bound is a consequence of the Erd\H{o}s-Stone-Simonovits Theorem.
 The lower bound is obtained by considering $T_{r-1} (\infty)$.
\end{proof}

\vspace{.5em}

In order to get more interesting problems we will assume our forbidden graph $F$  is bipartite and contains a cycle.
Given a bipartite graph $F$, we say that $\alpha$ is the \emph{exponent} of $F$ if there exists positive constants $c_1$ and $c_2$ such that for all
sufficiently large $n$, $c_1 n^{ \alpha } \leq \textup{ex}(n, F) \leq c_2 n^{ \alpha }$.  A special case of a conjecture of
Erd\H{o}s and Simonovits \cite{es} is that every bipartite graph has an exponent.
Motivated by the definition of $p(k)$, we define the \emph{infinite Tur\'{a}n number of $F$} to be
\[
\textup{ex}( \infty , F) := \sup_{G \in \mathcal{G}_{F}} \left\{   \liminf_{n \rightarrow \infty} \frac{ e(G_n) }{n^{\alpha} } \right \}
\]
where $\alpha$ is the exponent of $F$.  When $\alpha$ is the exponent of $F$ then an easy argument shows that
\[
\sup_{G \in \mathcal{G}_{F}} \left\{   \limsup_{n \rightarrow \infty} \frac{ e(G_n) }{n^{\alpha} } \right \}
= \limsup_{n \rightarrow \infty} \frac{ \exf }{n^{ \alpha} }
\]
which reduces this infinite problem to the finite one.
Replacing the $\limsup$ with the $\liminf$ leads to new problems that seem to be difficult.

The simplest bipartite graph with a cycle is $K_{2,2}$ and in this case, we have the well known asymptotic formula
$\textup{ex}(n, K_{2,2}) = \frac{1}{2}n^{3/2} + O(n^{5/4})$.  More generally, a construction of F\"{u}redi \cite{f} and an upper bound of
 K\"{o}vari, S\'{o}s, and Tur\'{a}n \cite{kst} shows that for any integer $t \geq 1$,
\[
\textup{ex}(n, K_{2,t+1} ) = \frac{1}{2} \sqrt{t} n^{3/2} + O(n^{5/4}).
\]
From this result we deduce that
\[
0 \leq \textup{ex}(\infty , K_{2,t+1} ) \leq \frac{1}{2} \sqrt{t}.
\]

Our first theorem improves these bounds.

\begin{theorem}\label{main theorem}
If $t \geq 1$ is an integer then
\[
0.23 \sqrt{t} < \textup{ex}( \infty , K_{2,t+1} ) \leq \left( \frac{\sqrt{13}}{14} ( \sqrt{8} - 1) \right) \sqrt{t} < 0.471 \sqrt{t}.
\]
\end{theorem}

When $t = 1$ the upper bound can be improved using a Maple program.

\begin{theorem}\label{c4ub}
The infinite Tur\'{a}n number of $K_{2,2}$ satisfies
\[
\textup{ex}( \infty , K_{2,2} ) \leq 0.41.
\]
\end{theorem}

Theorem~\ref{main theorem} shows that an
infinite $K_{2,t+1}$-free graph cannot always be as dense as a finite extremal $K_{2,t+1}$-free graph.  We believe that this holds in general.

\begin{conjecture}\label{conj}
Let $F$ be a bipartite graph which contains a cycle.  If $\textup{ex}(n , F) = c_F n^{ \alpha } + o( n^{\alpha} )$
then
\[
\textup{ex}( \infty , F) < c_F.
\]
\end{conjecture}

The assumption that $F$ contains a cycle is necessary.  If $\textup{ex}(n, F) = c_F n + o(n)$ then one can easily prove
$\textup{ex}( \infty , F) = c_F$.

Section 2 contains the proof of the upper bound of Theorem~\ref{main theorem} and the proof of Theorem~\ref{c4ub}.  Section 3 contains the proof of the lower bound of Theorem~\ref{main theorem} and Section 4 contains some concluding remarks.


\section{Proof of the upper bounds}


Given  integers $n, j \geq 1$, let $A_j = \{ 1 + (j-1)n , 2 + (j-1)n , \dots , jn \}$.
If $G$ is an infinite $K_{2,t+1}$-free graph with $V(G) = \nn$, let
$a_j n^{3/2}$ be the number of edges with both endpoints in $A_j$ and for $1 \leq i < j$, let
$b_{ij}n^{3/2}$ be the number of edges of $G$ with one endpoint in $A_i$ and the other in $A_j$.

\begin{lemma}\label{lm:1}
Let $t \geq 1$ be an integer and let $G$ be an infinite $K_{2,t+1}$-free graph with
\[
\liminf_{n \rightarrow \infty} \frac{ e(G_n) }{n^{3/2}} = c+ \epsilon
\]
where $c$ and $\epsilon$ are positive real numbers.
For any integer $k \geq 2$ and $\delta > 0$, there is an $n$, depending on $\epsilon$, $\delta$, $k$, and $t$ such that if
$A_j = \{ 1 + (j-1)n,2+(j-1)n, \dots , jn \}$ for $j \geq 1$, the non-negative real numbers $\{a_i \}_{1 \leq i \leq k}$ and $\{ b_{ij} \}_{ 1 \leq i < j \leq k }$ satisfy  the following
$2k$ inequalities:
\begin{enumerate}
\item $\displaystyle\sum_{i=1}^{l} a_i + \displaystyle\sum_{1 \leq i < j \leq l} b_{ij} \geq c l^{3/2}$ for $l =1, 2, \dots , k$,
\item $t + \delta \geq \displaystyle\sum_{l=1}^{i-1} b_{li}^{2} + 4a_{i}^{2} + \displaystyle\sum_{l=i+1}^{k} b_{il}^{2}$ for $l=1,2, \dots ,k$.
\end{enumerate}
\end{lemma}
\begin{proof}
If $B \subset \mathbb{N}$ is a finite set and $j \in \mathbb{N}$, let $d_B(j)$ be the number of
neighbors of $j$ in $B$. Since $G$ is $K_{2,t+1}$-free, each pair of vertices in $A_i$ have at most $t$ common neighbors thus for $1 \leq i \leq k$,
\begin{eqnarray*}
\frac{ t n^2}{2} \geq t{n \choose 2} & \geq &
\sum_{l =1}^{i-1} \sum_{j \in A_l} { d_{A_i}(j) \choose 2} + \sum_{j \in A_i} { d_{A_i}(j) \choose 2} +
\sum_{l=i+1}^{k} \sum_{j \in A_l} { d_{A_i} (j) \choose 2} \\
& \geq & n \left( \sum_{l=1}^{i-1} { b_{li} n^{1/2} \choose 2} + {2a_{i} n^{1/2} \choose 2} +
\sum_{l=i+1}^{k} { b_{il} n^{1/2} \choose 2} \right) \\
& = & \frac{n^2}{2} \left( \sum_{l=1}^{i-1} b_{li}^{2} + 4a_{i}^{2} + \sum_{l=i+1}^{k} b_{il}^{2} \right) - \frac{n^{3/2}}{2} \left( \sum_{l=1}^{i-1} b_{li} + 2a_{i} + \sum_{l=i+1}^{k} b_{il} \right).
\end{eqnarray*}
The graph $G_{kn}$ has $kn$ vertices and is $K_{2,t+1}$-free so
\[
 n^{3/2} \left( \sum_{l=1}^{i-1} b_{li} + a_i + \sum_{l=i+1}^{k} b_{il} \right) \leq e(G_{kn}) \leq 2 \sqrt{t}(kn)^{3/2}.
\]
We also have $a_i n^{3/2} \leq \sqrt{t} n^{3/2}$ for $1 \leq i \leq k$ so adding $a_i n^{3/2}$ to the previous inequality gives
\[
 n^{3/2}\left(\sum_{l=1}^{i-1} b_{li} + 2a_i + \sum_{l=i+1}^{k} b_{il}\right) \leq 3 n^{3/2}\sqrt{t} k^{3/2}.
\]
Given $\delta > 0$,
choose $n_0 = n_0 ( k , t , \delta )$ large enough so that for $n \geq n_0$,
\[
\frac{3 n^{3/2}\sqrt{t} k^{3/2}}{2} \leq \delta n^2.
\]
Combining this inequality with the first inequality we get for $1 \leq i \leq k$,
\[
t + \delta \geq \sum_{l=1}^{i-1} b_{li}^2 + 4a_i^2 + \sum_{l=i+1}^{k} b_{il}^2
\]
which is the second system of inequalities of the lemma.

Since $\liminf_{n \rightarrow \infty} \frac{ e(G_n) }{n^{3/2} } = c + \epsilon$, we can choose $n_1 = n_1 ( \epsilon )$
so large that for all $n \geq n_1$, $e(G_n) \geq c n^{3/2}$.  For each $1 \leq l \leq k$,
\[
c (ln)^{3/2} \leq e(G_{ln} ) = \left( \sum_{i=1}^{l} a_i + \sum_{1 \leq i < j \leq l } b_{ij} \right) n^{3/2}.
\]
For $n = \max \{ n_0 , n_1 \}$ and $A_j = \{ 1 + (j-1)n  ,2+ (j-1)n, \dots ,  jn \}$ where $j \in \mathbb{N}$, the asserted system of $2k$ inequalities holds.
\end{proof}

\vspace{.5em}

\begin{proof}[Proof of Theorem~\ref{main theorem}] Suppose there exists an infinite graph $G$ that is $K_{2,t+1}$-free and
\begin{equation*}
\liminf_{n \rightarrow \infty} \frac{ e(G_n) }{n^{3/2}} = \left( \frac{\sqrt{13}}{14} ( \sqrt{8} - 1) \right) \sqrt{t} + \epsilon
\end{equation*}
where $\epsilon > 0$.  Let
$c_t = \frac{ \sqrt{13} }{14} ( \sqrt{8} - 1 ) \sqrt{t}$.    By Lemma~\ref{lm:1}, there is an $n$ such that if
$A_1 = \{1,2, \dots , n\}$, $A_2 = \{ n+1 ,n+2 , \dots ,  2n\}$ and $\alpha = c_t + \frac{ \epsilon }{2}$ then
\[
a_1 \geq \alpha, ~~~~a_1 + a_2 + b_{12} \geq \alpha \sqrt{8}, ~~~~t + \delta \geq 4a_1^2 + b_{12}^2, ~~~~
\mbox{and}~~~~ t + \delta \geq 4a_2^2 + b_{12}^2~~~~~ (\ast)
\]
where $\delta > 0$ will be chosen later ($\delta$ will only depend on $\epsilon$ and $t$). The following claim completes the proof of Theorem \ref{main theorem}.

{\bf{Claim:}} There is a choice of $\delta$ such that there is no set of real numbers $a_1,a_2,b_{12}$ that
that satisfy the four inequalities of $(\ast)$.

Suppose there is a solution to $(\ast)$.  The first two inequalities imply
\[
2a_1 + a_2 + b_{12} \geq \alpha ( 1 + \sqrt{8} ).
\]
Consider the optimization problem:
\begin{align*}
 \textrm{maximize: } & f(x,y,z) = 2x + y + z\\
\textrm{subject to: } & g(x,y,z) = t + \delta - 4x^2 - z^2 \geq 0,\\
                     & h(x,y,z) = t + \delta - 4y^2 - z^2 \geq 0,\\
                     & x,y,z \geq 0.
\end{align*}
If $g(x,y,z) > 0$ then we can increase $x$, which increases $f$, until we obtain $g(x,y,z) = 0$.  Similarly if $h(x,y,z) > 0$ then
we can increase $y$, which again increases $f$, until $h(x,y,z) = 0$.  Therefore we may assume
that $g(x,y,z) = h(x,y,z) = 0$ since we are looking for the maximum value of $f$.  In this case, $x=y$ and the problem reduces to the simpler optimization problem:
\begin{align*}
 \textrm{ maximize: } &  f(x,z) = 3x + z\\
 \textrm{ subject to: } & g(x,z) =  4x^2 + z^2 = t + \delta,\\
                       & x,z \geq 0.
\end{align*}
A Lagrange Multiplier argument gives $f(x, z) \leq 13 \sqrt{ \frac{ t + \delta }{52} }$.
Recalling $\alpha = c_t + \frac{ \epsilon }{2}$, the inequalities $a_1 \geq \alpha$ and $a_1 + a_2 + b_{12} \geq \alpha \sqrt{8}$ imply
$2a_1 + a_2 + b_{12} \geq \alpha ( 1 + \sqrt{8} )$ so
\begin{equation}\label{infinite upper bound eq 1}
(c_t + \frac{ \epsilon }{2} ) (1 + \sqrt{8} ) \leq 13 \sqrt{ \frac{t+ \delta }{52} }
\end{equation}
The constant $c_t$ is defined so that $c_t (1 + \sqrt{8} ) = 13 \sqrt{ \frac{t}{52} }$ thus (\ref{infinite upper bound eq 1}) simplifies
to
\[
13 \sqrt{ \frac{t}{52} } + \frac{ \epsilon }{2} ( 1 + \sqrt{8} ) \leq 13 \sqrt{ \frac{ t + \delta }{52} }.
\]
If $\delta < 52 \left(  \sqrt{ \frac{t}{52} } + \frac{ \epsilon }{26} ( 1 + 2^{3/2} ) \right)^2 - t$ then this inequality is false
which completes the proof of the claim and the upper bound of Theorem \ref{main theorem}.
\end{proof}

\vspace{.5em}

\begin{proof}[Proof of Theorem~\ref{c4ub}] Using a Maple program, it was checked that there is no set of non-negative real numbers
$\{a_i \}_{1 \leq i \leq 30 }$ and $\{b_{ij} \}_{1 \leq i < j \leq 30}$ that satisfy the system of 60 inequalities when $t=1$,
$c = 0.41$, and $\delta = 0.00000001$.  The Maple program took less than ten minutes to verify this.
If we change the value of $c$ to $c = 0.4$ then the program finds a feasible solution rather quickly.
\end{proof}


\section{Proof of the lower bound}


For the lower bound we will use the $K_{2,t+1}$-free graphs constructed by F\"{u}redi \cite{f}.
Before going into the details we take a moment to informally describe the construction.  We will take an infinite sequence of
F\"{u}redi graphs where the first graph in the sequence has $n$ vertices for some large $n$, and for $j \geq 1$, the $j$-th graph in the sequence has $c^{j-1}n$ vertices where $c=3.58$ is chosen to optimize a certain function.  There will be no edges between distinct
F\"{u}redi graphs.  Now we proceed to the details.

Fix an integer $t \geq 1$ and let $q$ be a prime power with $q-1$ divisible by $t$.
Let $\mathbb{F}_q$ be the finite field with $q$ elements and let $g$ be an element of order $t$ in the multiplicative
group $\mathbb{F}_q^*$.  Let $X = \{1 , g , g^2 , \dots , g^{t-1} \}$ be the subgroup
of $\mathbb{F}_q^*$ generated by $g$.  Two pairs $(a,b) , (c,d) \in \mathbb{F}_q \times \mathbb{F}_q \backslash \{ (0,0) \}$ are equivalent if $(g^i a , g^i b) = (c,d)$ for some $0 \leq i \leq t-1$.  Write $\langle a , b \rangle$ for the equivalence class of
$(a,b)$ under this relation.  The vertices of the graph $H_{q,t}$ are the equivalence classes $\langle a,b \rangle$. A vertex  $\langle a , b \rangle$ is adjacent to another vertex $\langle x , y \rangle$ if
$ax + by \in X$.  The graph $H_{q,t}$ has loops and a loop contributes 1 to the degree of a vertex.  $H_{q,t}$ is
$q$-regular and has $\frac{q^2 - 1}{t}$ edges.  As far as we know, the eigenvalues of $H_{q,t}$ have not been computed explicitly.  With some work we could determine the eigenvalues of $H_{q,t}$ but for our purposes it is enough to show that the
eigenvalues of $H_{q,t}$ are contained in the set $\{q , \pm \sqrt{q} , \pm 1 \}$.

\begin{lemma}\label{evalue}
Let $A$ be the adjacency matrix of  $H_{q,t}$.  The largest eigenvalue of $A$ is $q$ with multiplicity 1, and the other
eigenvalues of $A$ are contained in the set $\{\pm \sqrt{q}$,  $\pm 1\}$.
\end{lemma}
\begin{proof}
Buchsbaum, Giancarlo, and Racz \cite{bgr}, proved that the vertices of $H_{q,t}$ can be partitioned into $q+1$ classes where each class contains
$\frac{q-1}{t}$ vertices, any two vertices in the same class have no common neighbors, and any two vertices in different classes have exactly $t$ common neighbors.  If $A$ is the adjacency matrix of $H_{q,t}$ whose columns are ordered according to this partition then
\[
A^2 =
  \begin{pmatrix}
   B & T & \cdots & T \\
   T & B & \cdots & T \\
   \vdots  & \vdots  & \ddots & \vdots  \\
   T & T & \cdots & B
  \end{pmatrix}
\]
where $T$ is the $\frac{q-1}{t} \times \frac{q-1}{t}$ matrix consisting of all $t$'s, and $B = qI$ where
$I$ is the $\frac{q-1}{t} \times \frac{q- 1}{t}$ identity matrix.  The all 1's vector is an eigenvector of $A^2$ with eigenvalue
$q^2$.

Let $v$ be the $\frac{q^2 - 1}{t}\times 1$ vector that is $1$ in the first $\frac{q-1}{t}$ coordinates and 0's elsewhere.
For $1 \leq i \leq q$, let $v_i$ be the $\frac{q^2 - 1}{t} \times 1$ vector that is 1 in position $\left( \frac{q-1}{t} \right) i +1$ through $2\left( \frac{q-1}{t} \right)i$ and 0's elsewhere.  For each $1 \leq i \leq q$, the vector $v + v_i$ is an eigenvector of $A^2$ with eigenvalue $1$.

For $1 \leq s < t \leq \frac{q^2 - 1}{t}$, let $v_{st}$ be the vector with a 1 in position $s$, a 1 in position $t$, and 0's elsewhere.
If $s = i( \frac{q-1}{t} ) + 1$ where $i \in \{0 ,1, \dots , q \}$ and
$t = i( \frac{q-1}{t} ) + j$ where $2 \leq j \leq \frac{q-1}{t}$ then $v_{st}$ is an eigenvector of $A^2$ with eigenvalue
$q$.  Altogether there are $(q+1)( \frac{q-1}{t} - 1)$ eigenvectors of this form.

The eigenvalues of $A^2$ are $q^2$ with multiplicity 1, $q$ and 1. The lemma follows from the fact that $H_{q,t}$ is connected and non-bipartite so that $q$ is the largest eigenvalue of $A$ and has multiplicity 1.
\end{proof}

\vspace{.5em}

The following proposition is well known (see Corollary 9.2.6 of \cite{as}).

\begin{proposition}\label{spectral}
Let $G$ be an $n$-vertex $d$-regular graph and assume that the absolute values of the eigenvalues of $G$ except for the first one is at most $\lambda$.  If $B \subset V(G)$ with $|B| = b n$ then
\[
\left| e(B) - \frac{1}{2}b^2 dn \right| \leq \frac{1}{2} \lambda bn.
\]
\end{proposition}

Given $0 \leq \epsilon \leq 1$, if $B$ is any set of $\epsilon \frac{q^2 - 1}{t}$ vertices of $H_{q,t}$ then
\[
e(B) \geq \frac{q(q^2 - 1)}{2t} \epsilon^2 - \frac{1}{2} q^{5/2} \epsilon - \epsilon \frac{q^2 - 1}{t} \geq \frac{q(q^2 - 1)}{2t} \epsilon^2 -  \frac{3}{2} q^{5/2} \epsilon
\]
since $B$ contains at most $\epsilon \frac{q^2 - 1}{t}$ vertices with loops.   	

\begin{lemma}\label{gadget}
Fix $t \geq 1$.  There is an integer $n_0$ such that for all $n \geq n_0$, there is an $n$-vertex graph with
$V(G) = \{1,2, \dots , n \}$ such that $G$ is $K_{2,t+1}$-free and for any $0 < \epsilon \leq 1$, the subgraph
$G[ \{1,2, \dots, \epsilon n \}]$ has at least
\[
\frac{\epsilon^2}{2} \sqrt{t}n^{3/2} - 7t^{5/4}n^{4/3}
\]
edges.
\end{lemma}
\begin{proof}
By \cite{hi}, we can choose $n_0$ so large that for every $n \geq n_0$ there is a prime $p$ with $p \equiv 1 (\textup{mod}~t)$ and
\begin{equation}\label{condition}
\sqrt{nt} - n^{1/3} \leq p \leq \sqrt{nt}.
\end{equation}
Choose a prime $p$ that satisfies (\ref{condition}).  Let $H_{p,t}$ be the $K_{2,t+1}$-free F\"{u}redi
graph on $\frac{p^2 - 1}{t}$ vertices with $\frac{p}{2t} (p^2- 1)$ edges.  Let
$e(n) = n - \frac{p^2 - 1}{t}$.  A short calculation shows that (\ref{condition}) implies
\[
0 \leq e(n) \leq \frac{2n^{5/6}}{ \sqrt{t}}.
\]
Arbitrarily label the vertices
of $H_{p,t}$ with $\{1,2, \dots , \frac{p^2-1}{t} \}$ and let
$\{ \frac{p^2 - 1}{t} +1 , \frac{p^2 - 1}{t} + 2 , \dots , n \}$ be the vertex set of $\overline{K_{e(n)}}$, an empty graph on $e(n)$ vertices.  Let $G_n = H_{p,t} \cup \overline{K_{e(n)}}$ and let $0 < \epsilon \leq 1$.

If $\epsilon n \leq \frac{p^2 - 1}{t}$ then by Proposition~\ref{spectral},
\[
e(G [ \{1, 2, \dots , \epsilon n \} ]) \geq \frac{\epsilon^2}{2} \sqrt{t} n^{3/2} - 7t^{5/4} n^{4/3}.
\]
If $\epsilon n  >\frac{p^2 - 1}{t}$ then
\[
e( G [ \{1,2, \dots , \epsilon n \} ]) = e( H_{p,t} ) \geq \frac{ \sqrt{t}}{2} n^{3/2} - 7t^{5/4}n^{4/3}.
\]
\end{proof}

\vspace{.5em}

We are now ready to define an infinite $K_{2,t+1}$-free graph that always has many edges.  Fix an integer $t \geq 1$.
Let $n_0$ be the integer whose existence is guaranteed by
Lemma~\ref{gadget} and fix an integer $n \geq n_0$.  Let $n_j = c^{j-1}n$ where $c >1$ is a fixed positive constant that will be determined later.
For each $j \geq 1$, let $G^{(n_j)}$ be the $K_{2,t+1}$-free graph with $n_j$ vertices of Lemma~\ref{gadget}.  Consider the sequence of graphs $\{ G^{(n_j)} \}_{j=1}^{\infty}$.
Label the vertices of $G^{(n_1)}$ with $\{1,2, \dots , n_1 \}$ using each label exactly once, then label the vertices of
$G^{ (n_2)}$ with $\{n_1 + 1 , n_1 +2 , \dots , n_1 + n_2 \}$ again using each label exactly once, and so on.  This defines an
infinite $K_{2,t+1}$-free graph $G$ with $V(G) = \nn$.

\begin{lemma}\label{lb}
There is an integer $N_0$ such that for any $N \geq N_0$
\[
\frac{ e ( G_{N} )}{N^{3/2}} \geq  0.23 \sqrt{t}.
\]
\end{lemma}
\begin{proof}
Fix an integer $N_0 \geq  Cn$,  where $C$ is a sufficiently large number which will be specified later,  and  $n \geq n_0$ is the fixed integer specified in the paragraph preceding the statement of
Lemma~\ref{lb}.
Let $N \geq N_0$ and write $N = n_1 + n_2 + \dots + n_{j} + \epsilon n_{j+1}$ where $j \geq 1$, $0 \leq \epsilon < 1$, and $n_j= c^{j-1}n$.
The number of vertices of $G_{N}$ is
\[
 n_1 + n_2 + \dots + n_{j} + \epsilon n_{j+1} = n + cn + \dots + c^{j-1}n + \epsilon c^{j}n = n \left( \frac{c^j -1}{c-1} + \epsilon c^j \right)
\]
and
\begin{eqnarray*}
e(G_{N} ) & \geq &  \sum_{i=1}^{j} \left( \frac{ \sqrt{t} }{2} (c^{i-1}n)^{3/2} - 7t^{5/4}(c^{i-1}n)^{4/3} \right) +
\frac{ \epsilon^2 \sqrt{t} }{2} (c^{j} n)^{3/2} - 7 t^{5/4} ( c^{j} n )^{4/3} \\
& \geq &
\frac{ \sqrt{t} n^{3/2} }{2} \left(  \frac{ c^{3j/2} - 1 }{c^{3/2} - 1} + \epsilon^2 c^{3j/2} \right) - \frac{ 7t^{5/4}n^{4/3}
c^{4j/3 } c^{4/3}}{c^{4/3} - 1}.
\end{eqnarray*}
Computing the ratio of $e(G_{N} )$ to $N^{3/2}$ we get
\begin{equation}\label{ratio}
\frac{ e(G_{N} ) }{N^{3/2}} \geq \frac{ \sqrt{t}}{2} \left(  \frac{   \frac{1-c^{-3j/2}}{c^{3/2}-1} + \epsilon^2 }
{ (   \frac{1 - c^{-j}}{c-1} + \epsilon )^{3/2} } \right) - \frac{7c^{4/3}t^{5/4}}{n^{1/6}c^{j/6} (  \frac{1-c^{-j}}{c-1} )^{3/2} }.
\end{equation}

Since $n$ and $t$ are fixed, $0 \leq \epsilon \leq 1$,  and $c > 1$,  we can find a large integer $j_0$  such that
\[
\left| \frac{ \sqrt{t}}{2} \left(  \frac{   \frac{1-c^{-3j/2}}{c^{3/2}-1} + \epsilon^2 }
{ (   \frac{1 - c^{-j}}{c-1} + \epsilon )^{3/2} } \right) - \frac{7c^{4/3}t^{5/4}}{n^{1/6}c^{j/6} (  \frac{1-c^{-j}}{c-1} )^{3/2} }- \frac{ \sqrt{t}}{2}  \left(  \frac{   \frac{1}{c^{3/2} - 1} + \epsilon^2 }{  ( \frac{1}{c-1} + \epsilon )^{3/2} } \right) \right| \leq 0.001
 \]
for each $j \geq j_0$.
We want to find a value of $c$ such that the function
\[
f(c,\epsilon) =  \frac{   \frac{1}{c^{3/2} - 1} + \epsilon^2 }{  2( \frac{1}{c-1} + \epsilon )^{3/2} }
\]
has a large minimum value over $0 \leq \epsilon \leq 1$.  Finding the exact $c$ can be done but the expression is rather complicated so instead we choose $c = 3.58$.    Using elementary calculus one can check that
the minimum value of $f(3.58 , \epsilon)$ over $0 \leq \epsilon \leq 1$ is strictly greater than $0.2306$.
Choose $C=3.58^{j_0+2}$ to complete the proof of  the lower bound of
Theorem~\ref{main theorem}.
\end{proof}


\section{Concluding remarks}


In this paper we proved upper and lower bounds on the infinite Tur\'{a}n number for $K_{2, t+1}$.  Specializing to
$\textup{ex}(n, K_{2,2})$ we were able to obtain a better upper bound.  It is possible that one can improve the upper bounds
by taking $k$ very large and then analyzing the associated system of inequalities but it is not clear whether or not this method will give an upper bound that is close to the lower bound.  The authors believe that the true value of
$\textup{ex}( \infty , K_{2, t+1} )$ should be closer to the lower bound.

Using the method of Section 3, one can obtain a lower bound on $\textup{ex}(\infty , K_{3,3} )$.  The projective norm graphs
constructed by Koll\'{a}r, R\'{o}nyai, Szab\'{o} \cite{krs} (see also Alon, R\'{o}nyai, Szab\'{o} \cite{ars}) are
of $K_{t, (t-1)! + 1}$-free graphs and Szab\'{o} \cite{s} computed the eigenvalues of these graphs.  Making the appropriate changes to the proof of the lower bound of Theorem~\ref{main theorem}, one can show
\[
\textup{ex}( \infty , K_{3,3} ) \geq 0.214.
\]
One difference is that we choose $c = 5.49$ instead of $c = 3.58$.
The same counting used to prove the upper bound of Theorem~\ref{main theorem} can be adapted to give a
system of inequalities similar to the one of Lemma~\ref{lm:1} which must be satisfied by an infinite $K_{3,3}$-free graph.
Using our Maple program we can show that $\textup{ex}( \infty , K_{3,3} ) \leq 0.46$ which supports
Conjecture~\ref{conj}.


\end{document}